\documentclass[11pt]{article}

\usepackage{amsmath}
\usepackage{amsfonts}
\usepackage{amssymb}
\usepackage{amsthm}
\usepackage{enumitem}
\usepackage{mathtools}
\usepackage{setspace}
\usepackage{etoolbox}

\newtheoremstyle{mystyle}
{11pt}                          
{11pt}                          
{}                                      
{}                                      
{\bfseries}                     
{}                                      
{5.5pt}                         
{}                                      

\theoremstyle{mystyle}
\newtheorem{theorem}{Theorem}[section]

\newtheorem{lemma}[theorem]{Lemma}
\newtheorem{proposition}[theorem]{Proposition}
\newtheorem{corollary}[theorem]{Corollary}

\renewenvironment{proof}[1][Proof.]{\vspace{-16.5pt} \begin{trivlist}
        \item[\hskip \labelsep {\bfseries #1}]}{\qed \end{trivlist}}

\setitemize{noitemsep, topsep=5.5pt, parsep=5.5pt, partopsep=0pt}

\allowdisplaybreaks
\appto\normalsize{
        \abovedisplayskip=5.5pt plus 2pt minus 2pt
        \belowdisplayskip=5.5pt plus 2pt minus 2pt
        \abovedisplayshortskip=5.5pt plus 2pt minus 2pt
        \belowdisplayshortskip=5.5pt plus 2pt minus 2pt}
\appto\small{
        \abovedisplayskip=5.5pt plus 2pt minus 2pt
        \belowdisplayskip=5.5pt plus 2pt minus 2pt
        \abovedisplayshortskip=5.5pt plus 2pt minus 2pt
        \belowdisplayshortskip=5.5pt plus 2pt minus 2pt}

\newcommand{\gap}{\vspace{11pt}}

\newcommand{\R}{\mathcal{R}}

\newcommand{\Rn}{\mathcal{R}^n}

\newcommand{\Sn}{\mathcal{S}^n}

\newcommand{\Hn}{\mathcal{H}^n}
\newcommand{\V}{{\cal V}}
\newcommand{\wtv}{\widetilde{\V}}
\newcommand{\wtt}{\widetilde{T}}

\title{\bf  A Riesz-Thorin type interpolation theorem \\
in Euclidean Jordan algebras}
\author{
        M. Seetharama Gowda\\
        Department of Mathematics and Statistics\\
        University of Maryland, Baltimore County\\
        Baltimore, Maryland 21250, USA\\
        gowda@umbc.edu\\
and\\
Roman Sznajder\\
Department of Mathematics\\
Bowie State University\\
Bowie, Maryland 20715, USA\\
rsznajder@bowiestate.edu
}

\date{\today}

\begin{document}

\maketitle

\begin{abstract}
In a  Euclidean Jordan algebra $\V$ of rank $n$ which carries the trace inner product, to each element 
$a$ we associate the eigenvalue vector $\lambda(a)$ in $\Rn$ whose components are the eigenvalues of $a$ written 
in the decreasing order. For  any $p\in [1,\infty]$, we define the spectral $p$-norm of $a$ to be the $p$-norm of 
$\lambda(a)$ in $\R^n$.  
In a recent paper,  based on the $K$-method of real interpolation theory and a majorization technique, we described an interpolation theorem for a linear transformation on $\V$ relative to the same spectral norm. In this paper, using standard 
complex function theory methods, we describe a Riesz-Thorin type interpolation theorem relative to two different spectral norms.
We illustrate the result by estimating the norms of certain special linear transformations such as Lyapunov transformations, quadratic representations, and positive transformations.
\end{abstract}

\vspace{1cm}
\noindent{\bf Key Words:}
Euclidean Jordan algebra, Riesz-Thorin type interpolation  theorem
\\

\noindent{\bf AMS Subject Classification:}
15A18,  15A60, 17C20, 47A57
\newpage

\section{Introduction}
Consider a Euclidean Jordan algebra $\V$ of rank $n$ which carries the trace inner product. To each element $a$ in $\V$, we 
associate the eigenvalue vector $\lambda(a)$ whose components are the eigenvalues of $a$ written
in the decreasing order. For any $p\in [1,\infty]$, we define the spectral $p$-norm on $\V$ by
$$||a||_p:=||\lambda(a)||_p,$$
where the right-hand side is the usual $p$-norm of the vector $\lambda(a)$ in $\Rn$. Given $r,s\in [1,\infty]$ and a  linear transformation $T:\V\rightarrow \V$, we let 
$$||T||_{r\rightarrow s}:=\sup_{a\neq 0} \frac{||T(a)||_s}{||a||_r}.$$
In \cite{gowda-holdertype}, based on the $K$-method of real interpolation theory \cite{lunardi}, the following result was proved.

\begin{theorem}\label{diagonal interpolation theorem}
{\it
Suppose $1\leq r,s,p\leq \infty$,  $0\leq \theta\leq 1$, and
\begin{equation}\label{r,s,p}
\frac{1}{p}=\frac{1-\theta}{r}+\frac{\theta}{s}.
\end{equation}
Then, for  any linear transformation $T:\V\rightarrow \V$,
\begin{equation}\label{diagonal interpolation inequality}
||T||_{p\rightarrow p}\leq ||T||_{r\rightarrow r}^{1-\theta}\,\,||T||_{s\rightarrow s}^{\theta}.
\end{equation}
In particular,
\begin{equation}\label{special diagonal case}
||T||_{p\rightarrow p}\leq ||T||_{\infty\rightarrow \infty}^{1-\frac{1}{p}}\,\,||T||_{1\rightarrow 1}^{\frac{1}{p}}.
\end{equation}
}
\end{theorem}
A key idea in the  proof of  the above result is the use of a majorization result that connects a
$K$-functional defined on $\V$ with a $K$-functional on an $L_p$-space.
In \cite{gowda-holdertype}, the issue of proving an inequality of the type (\ref{diagonal interpolation inequality}) that deals with the norm of $T$ relative to two spectral norms (such as $||T||_{r\rightarrow s}$) was raised. In the present paper, based on  standard complex function theory methods (especially, Hadamard's three lines theorem) we prove the following  Riesz-Thorin type interpolation result.

\begin{theorem}\label{interpolation theorem}
{\it Let $r_0,r_1,s_0,s_1\in [1,\infty]$ and $\theta\in [0,1]$. Consider 
$r_\theta$ and $s_\theta$ in $[1,\infty]$ defined by
$$\frac{1}{r_\theta}=\frac{1-\theta}{r_0}+\frac{\theta}{r_1}\quad \mbox{and}\quad \frac{1}{s_\theta}=\frac{1-\theta}{s_0}+\frac{\theta}{s_1}.$$  Then, for any linear transformation $T$ on $\V$,
\begin{equation}\label{full interpolation inequality}
||T||_{r_\theta\rightarrow s_\theta}\leq C\,||T||_{r_0\rightarrow s_0}^{1-\theta}\,||T||_{r_1\rightarrow s_1}^{\theta},
\end{equation}
where $C$ is a constant, $1\leq C\leq 4$,  that  depends only on $r_0,r_1,s_0,s_1$. }
\end{theorem}

Illustrating this result, we estimate the norms of some special linear transformations on $\V$ such as   Lyapunov transformations, quadratic representations, and positive transformations.  
\section{Preliminaries}
Throughout this paper $(\V, \circ,\langle\cdot,\cdot\rangle)$ denotes a Euclidean Jordan algebra of rank 
$n$ with unit element $e$ \cite{faraut-koranyi}, \cite{gowda-sznajder-tao}. We let letters $a,b,c,d$, and $v$  denote elements of $\V$, $x$ and $y$ denote elements of $\Rn$,  and  write $z$ for a complex variable.
For $a,b\in \V$, we denote their Jordan product and
inner product by $a\circ b$ and $\langle a,b\rangle$, respectively.
It is known that any Euclidean Jordan algebra is a direct product/sum 
of simple Euclidean Jordan algebras and every simple Euclidean Jordan algebra is isomorphic to one of five algebras, 
three of which are the algebras of $n\times n$ real/complex/quaternion Hermitian matrices. The other two are: the algebra of $3\times 3$ octonion Hermitian matrices and the Jordan spin algebra.

According to the {\it spectral decomposition 
theorem} \cite{faraut-koranyi}, any element $a\in \V$ has a decomposition
$$a=a_1e_1+a_2e_2+\cdots+a_ne_n,$$
where the real numbers $a_1,a_2,\ldots, a_n$ are (called) the eigenvalues of $a$ and 
$\{e_1,e_2,\ldots, e_n\}$ is a Jordan frame in $\V$. (An element may have decompositions coming from different Jordan frames, but the eigenvalues remain the same.) Then, $\lambda(a)$-- called the {\it eigenvalue vector} of $a$-- is the vector of eigenvalues of $a$ written in the decreasing order. The {\it trace  and spectral $p$-norm} of $a$ are defined by 
$$tr(a):=a_1+a_2+\cdots+a_n\quad\mbox{and}\quad ||a||_p:=||\lambda(a)||_p,$$
where $||\lambda(a)||_p$ denotes the usual $p$-norm of a vector in $\Rn$.
An element $a$ is said to be {\it invertible} if all its eigenvalues are nonzero. We note that the set of invertible elements is  dense in $\V$.
{\it Throughout this paper, we assume that the inner product is the trace inner product, that is, 
$\langle a,b\rangle=tr(a\circ b).$}

\gap

Given a spectral decomposition $a=\sum_{j=1}^{n} a_je_j$ and a real number $\gamma>0$, 
we write
\begin{equation} \label{power notation}
|a|:=\sum_{j=1}^{n} |a_j|e_j,\,\,|a|^\gamma:=\sum_{j=1}^{n} |a_j|^\gamma e_j\quad\mbox{and}\quad ||a||_1=\sum_{j=1}^{n}|a_j|=tr(|a|).
\end{equation}

In what follows, we say that  $q$ is the conjugate of   $p\in [1,\infty]$ if  $\frac{1}{p}+\frac{1}{q}=1$ and denote the conjugate of $r\in [1,\infty]$ by  $r^\prime$. Also, we use the standard convention that  $1/\infty=0$.
\\
 
Based on the Fan-Theobald-von Neumann type inequality \cite{baes}
$$\langle a,b\rangle \leq \langle \lambda(a),\lambda(b)\rangle\quad (a,b\in \V)$$ 
and majorization techniques, the following result was proved in \cite{gowda-holdertype}.

\begin{theorem}\label{some p,q norm inequalities}
{\it
Let $p\in [1,\infty]$ with conjugate $q$. Then the following statements hold in $\V$:
\begin{itemize}
\item [$(i)$] $|\langle a,b\rangle |\leq ||a\circ b||_1\leq ||a||_p\,||b||_q$.
\item [$(ii)$]
$\sup_{b\neq 0}\frac{|\langle a,b\rangle|}{||b||_q}=
||a||_p.$
\end{itemize}
}
\end{theorem}

\section{The proof of the  interpolation theorem}

The Riesz-Thorin interpolation theorem, stated in the setting of $L_p$-spaces,
 is well-known in classical analysis. There is also a Riesz-Thorin type result 
available for linear transformations
 on the space of complex $n\times n$ matrices with respect to Schatten $p$-norms, see the interpolation theorem of Calder\'{o}n-Lions (\cite{reed-simon}, Theorem IX.20). Our Theorem \ref{interpolation theorem} is stated in the setting of  Euclidean Jordan 
algebras relative to spectral norms. In the absence of an
 isomorphism type argument that immediately gives our result, we offer a  
 proof that mimics the classical proof  based on the Hadamard's three lines theorem of complex function theory 
(\cite{folland}, Theorem 6.27).
In the proof given below, we complexify the real inner product 
space $\V$ and  define norms on this complexification in such a way to 
have a H\"{o}lder type inequality. This procedure results in a constant $C$ 
in the Riesz-Thorin type inequality (\ref{full interpolation inequality}) that is different from $1$. Possibly, a different argument may show that this constant can be replaced by $1$. 
\\ 

Recall that  $a$ and $b$ denote elements of $\V$ and 
$z$ denotes a complex variable.  
For $x=(x_1,x_2,\ldots,x_n)$ and $y=(y_1,y_2,\ldots, y_n)$ in $\Rn$, we write $x+iy=(x_1+iy_1,x_2+iy_2,\ldots, x_n+iy_n)\in \mathbb{C}^n$. Let $T$ be a linear transformation on $\V$. We consider complexifications of $\V$ and $T$: 
$$\wtv:=\V+i\V\quad\mbox{and}\quad \wtt(a+ib):=T(a)+i\,T(b) \quad (a,b\in \V).$$
We define the inner product and spectral $p$-norm on $\wtv$ as follows. For $a,b,c,d\in \V$, 
$$\langle a+ib,c+id\rangle:=\Big [\langle a,c\rangle+\langle b,d\rangle\Big ]+i\Big [\langle b,c\rangle-\langle a,d\rangle\Big ]\quad\mbox{and}\quad ||a+ib||_p:=||a||_p+||b||_p.$$
It is easily seen that  
$\wtv$ is a complex inner product space, $\wtt$ is a (complex) linear transformation on $\wtv$. We state the following simple lemma.

\begin{lemma} \label{lemma}
Consider  $\wtv$ and $\wtt$  as above. Let $p\in [1,\infty]$ with conjugate $q$, and $r,s\in [1,\infty].$
Then,
\begin{itemize}
\item [$(i)$] $|\langle a+ib,c+id\rangle|\leq ||a+ib||_p\,||c+id||_q$ for all $a,b,c,d\in \V$, and
\item [$(ii)$] $||\wtt||_{r\rightarrow s}=||T||_{r\rightarrow s}.$
\end{itemize}
\end{lemma}

\begin{proof} $(i)$  By the definition of inner product  in $\wtv$ and Theorem \ref{some p,q norm inequalities}, 
$$|\langle a+ib,c+id\rangle|\leq |\langle a,c\rangle|+|\langle a,d\rangle|+|\langle b,c\rangle|+|\langle b,d\rangle|\leq 
||a||_p\,||c||_q+||a||_p\,||d||_q+||b||_p\,||c||_q+||b||_p\,||d||_q.$$
Since the right-hand side is $||a+ib||_p\,||c+id||_q,$ the stated inequality follows. 
\\
$(ii)$  For $a,b\in \V$,
$$||\wtt(a+ib)||_s=||T(a)+i\,T(b)||_s=||T(a)||_s+||T(b)||_s\leq ||T||_{r\rightarrow s}(||a||_r+||b||_r)=||T||_{r\rightarrow s}\,||a+ib||_r.$$
This implies that $||\wtt||_{r\rightarrow s}\leq ||T||_{r\rightarrow s}.$ The reverse inequality holds as 
$\wtt$ is an extension of $T$ to $\wtv$.
Hence we have $(ii)$.
\end{proof}

\gap

We now come to the proof of Theorem \ref{interpolation theorem}.
In what follows, {\it for any $p\in [1,\infty]$ with conjugate $q$, we let}
$$
{\it C_p=\left \{\begin{array}{cl}  \sqrt{2} & {\rm if}~ 1 \leq p \leq 2,\\2^{\frac{1}{q}} & {\rm if}~ 2\leq p \leq \infty\,. \end{array} \right . 
}
$$
\gap

\begin{proof} Let the assumptions of the theorem be in place. Recalling that 
  $s^\prime$ denotes the conjugate of (any) $s\in [1,\infty]$, we define
\begin{equation}\label{constant C}
C:=\max\{  C_{r_0}C_{s_0^\prime},  C_{r_1}C_{s_1^\prime}\}
\end{equation}
which is a number between $1$ and $4$, and depends only on 
$r_0,r_1,s_0,s_1$. We show that (\ref{full interpolation inequality}) holds for this $C$. Since (\ref{full interpolation inequality}) clearly holds when $\theta=0$ or $\theta=1$, from now on, we assume that   
$0<\theta<1.$

 Let   
$$\alpha_j:=\frac{1}{r_j},\,\,\beta_j:=\frac{1}{s_j},\,\,\mbox{and}\,\,M_j:=||T||_{r_j\rightarrow s_j}\quad (j=0,1),$$
$$\alpha:=\frac{1}{r_\theta},\,\beta:=\frac{1}{s_\theta},\,\,\mbox{and}\,\,M_\theta:=||T||_{r_\theta\rightarrow s_\theta},$$
and for a complex variable $z$,
$$\alpha(z):=(1-z)\alpha_0+z\alpha_1\quad\mbox{and}\quad \beta(z):=(1-z)\beta_0+z\beta_1.$$
We show that 
\begin{equation} \label{M estimate}
M_\theta\leq C\,M_0^{1-\theta}\,M_1^\theta.
\end{equation}

Now, using Theorem \ref{some p,q norm inequalities}, Item $(ii)$,
$$M_\theta=||T||_{r_\theta\rightarrow s_\theta}=\sup_{0\neq a\in \V}\frac{||T(a)||_{s_\theta}}{||a||_{r_\theta}}=
\sup_{0\neq a,b\in \V}
\frac{|\langle T(a),b\rangle|}{||a||_{r_\theta}||b||_{s_\theta^\prime}}=\sup_{||a||_{r_\theta}=1=||b||_{s_\theta^\prime}}
|\langle Ta,b\rangle|.$$
To prove (\ref{M estimate}), it is enough to show that for any $a$ and $b$ in $\V$ with 
$||a||_{r_\theta}=1=||b||_{s_\theta^\prime},$
\begin{equation}\label{secondary inequality}
|\langle Ta,b\rangle|\leq  C\,M_0^{1-\theta}\,M_1^\theta.
\end{equation} 
By continuity, it is enough to prove  this for $a$ and $b$ invertible (that is, with all their eigenvalues nonzero). 
We fix such $a$ and $b$ and write their spectral decompositions:
$$a=\sum_{j=1}^{n} |a_j|\varepsilon_je_j\quad\mbox{and}\quad b=\sum_{j=1}^{n}|b_j|\delta_j f_j,$$
where $\{e_1,e_2,\ldots, e_n\}$ and $\{f_1,f_2,\ldots, f_n\}$ are Jordan frames, $\varepsilon_j,\delta_j\in \{-1,1\}$ for all $j$, and $a_j$s are the eigenvalue of $a$, etc. 
Now, with the observation that $0<\alpha,\beta<1$, we define two elements in $\wtv$:
$$a_z:=\sum_{j=1}^{n} |a_j|^{\frac{\alpha(z)}{\alpha}} \varepsilon_je_j\quad\mbox{and}\quad b_z:=
\sum_{j=1}^{n} |b_j|^{\frac{1-\beta(z)}{1-\beta}}\delta_jf_j,$$
where we consider only the principal values while defining the exponentials.
Then the function 
$$\phi(z):=\langle \wtt(a_z),b_z\rangle$$ 
is continuous on the strip $\{z: 0\leq Re(z)\leq 1\}$ and analytic in its interior. \\
We estimate $|\phi(z)|$ on the lines $Re(z)=0$ and $Re(z)=1$ and then apply Hadamard's three lines theorem (\cite{folland}, Theorem 6.27). First,
 suppose  $Re(z)=0$. 
Let  
$$|a_j|^{\frac{\alpha(z)}{\alpha}}=x_j+i\,y_j,\,\,x:=(x_1,x_2,\ldots,x_n)\in \Rn,\,\,\mbox{and}\,\,
y:=(y_1,y_2,\ldots,y_n)\in \Rn.$$
 Then, $|x_j+iy_j|=\left | |a_j|^{\frac{\alpha(z)}{\alpha}}\right |=|a_j|^{\frac{\alpha_0}{\alpha}}$. When $r_0=\infty$, that is, when $\alpha_0=0$, $|x_j+iy_j|=1$ for all $j$ and hence (in $\mathbb{C}^n$), $||x+iy||_{r_0}=1$. When, $r_0<\infty$, 
$|x_j+iy_j|^{r_0}= |a_j|^{r_\theta}$. So, because $||a||_{r_\theta}=1$, we have
$||x+iy||_{r_0}^{r_0}=\sum_{j=1}^{n}|x_j+iy_j|^{r_0}=\sum_{j=1}^{n}|a_j|^{r_\theta}=1.$ Thus, in both cases, 
\begin{equation}\label{constraint}
||x+iy||_{r_0}=1.
\end{equation}
Now, $a_z=\sum_{j=1}^{n}(x_j+iy_j)\varepsilon_je_j=(\sum_{j=1}^{n}x_j\varepsilon_je_j)+i(\sum_{j=1}^{n}y_j\varepsilon_je_j)$ and so, 
$$||a_z||_{r_0}=||\sum_{j=1}^{n} x_j\varepsilon_je_j||_{r_0}+||\sum_{j=1}^{n}y_j\varepsilon_je_j||_{r_0}
=||x||_{r_0}+||y||_{r_0}.$$
In view of (\ref{constraint}), from Proposition \ref{optimization} in the Appendix, 
we have, $$||a_z||_{r_0}
\leq C_{r_0}.$$

Similarly, $||b_z||_{s_0^\prime}\leq C_{s_0^\prime}.$ 
Hence, when $Re(z)=0$, Lemma \ref{lemma} gives 
$$|\phi(z)|\leq ||\wtt(a_z)||_{s_0}\,||b_z||_{s_0^\prime}\leq ||\wtt||_{r_0\rightarrow s_0}\,||a_z||_{r_0}||b_z||_{s_0^\prime}\leq   ||T||_{r_0\rightarrow s_0}\,C_{r_0}\,C_{s_0^\prime}= C_{r_0}C_{s_0^\prime}\,M_0.$$ 

A similar computation shows that 
$$Re(z)=1\Rightarrow 
|\phi(z)|\leq  C_{r_1}C_{s_1^\prime}M_1.$$
By Hadamard's three lines theorem,  
$$|\phi(\theta)|\leq \Big ( C_{r_0}C_{s_0^\prime}\,M_0\Big )^{1-\theta}\,\Big(C_{r_1}C_{s_1^\prime}\,M_1\Big )^\theta.$$
We recall that  $C=\max\{  C_{r_0}C_{s_0^\prime},  C_{r_1}C_{s_1^\prime}\}.$
Now, $a_\theta=a$ and $b_\theta=b$, and so,
$\phi(\theta)=\langle T(a),b\rangle.$ Hence, 
$$|\langle T(a),b\rangle| \leq  C \,M_0^{1-\theta}\,M_1^\theta.$$
This gives (\ref{secondary inequality}) 
and the proof is complete. 
\end{proof}

\gap

\noindent{\bf Remarks.} Instead of the constant $C$ defined in (\ref{constant C}), one may consider a slightly better constant, namely, $\max\{  (C_{r_0}C_{s_0^\prime})^{1-\theta},  (C_{r_1}C_{s_1^\prime})^\theta\}.$
However, this constant depends on $\theta$. 
\\

We now consider the problem of estimating the norms of certain special linear transformations on $\V$ relative to spectral norms. First, we make two observations. 
Writing $T^*$ for the adjoint of a linear transformation $T$ on $\V$, we note, thanks to Theorem \ref{some p,q norm inequalities}, that
$$||T^*||_{r\rightarrow s}=||T||_{s^\prime\rightarrow r^\prime},$$
where $r^\prime$  denotes the conjugate of $r$, etc.
Also, knowing the norms $||T||_{1\rightarrow 1}$, $||T||_{\infty\rightarrow \infty}$, $||T||_{1\rightarrow p}$, and $||T||_{p\rightarrow 1}$, etc., one can estimate $||T||_{r\rightarrow s}$ for various $r$ and $s$. When $r=s$, (\ref{special diagonal case}) gives such an estimate. In the result below, we consider the case $r\neq s$.
 
\begin{corollary} \label{corollary}
{\it Let $1\leq r\neq s\leq \infty$. Then, for any linear transformation $T:\V\rightarrow \V$,
$$||T||_{r\rightarrow s}\leq \left \{\begin{array}{cl}  
2\sqrt{2}\,||T||_{\infty\rightarrow \infty}^{1-\frac{1}{r}}\,||T||_{1\rightarrow \frac{s}{r}}^{\frac{1}{r}}  & {\rm if}~ r <s,\\
2\sqrt{2}\,||T||_{\infty\rightarrow \infty}^{1-\frac{1}{s}}\,||T||_{\frac{r}{s}\rightarrow 1}^{\frac{1}{s}}  & {\rm if}~ r> s\,. \end{array} \right . $$
}
\end{corollary} 

\begin{proof}
The stated inequalities are obtained by specializing Theorem \ref{interpolation theorem}.
When $r< s$, we let 
$$r_0=\infty,\,s_0=\infty,\,r_1=1,\,s_1=\frac{s}{r},\,r_\theta=r,\,s_\theta=s,\,\mbox{and}\,\,\theta=\frac{1}{r}. $$ In this case,
$C=\max\{  C_{r_0}C_{s_0^\prime},  C_{r_1}C_{s_1^\prime}\}=2\sqrt{2}.$
When $r>s$, we let
$$r_0=\infty,\,s_0=\infty,\,r_1=\frac{r}{s},\,s_1=1,\,r_\theta=r,\,s_\theta=s,\,\mbox{and}\,\,\theta=\frac{1}{s}.$$ In this case also,
$C=2\sqrt{2}.$
\end{proof}

\gap

\noindent{\bf Remarks.} In the result above,  by considering $\max\{  (C_{r_0}C_{s_0^\prime})^{1-\theta},  (C_{r_1}C_{s_1^\prime})^\theta\}$, 
one can replace the 
constant $2\sqrt{2}$ by  the following:
\begin{center}
$ (2\sqrt{2})^{\max\{1-\frac{1}{r},\frac{1}{r}\}}$ when $r<s$ and 
$ (2\sqrt{2})^{\max\{1-\frac{1}{s},\frac{1}{s}\}}$ when $r>s$.
\end{center}

\gap

We now illustrate our results via some examples. 
For any $a\in \V$,  consider the {\it Lyapunov transformation} $L_a$ and the {\it quadratic representation} 
$P_a$ defined by 
$$L_a(v):=a\circ v\quad\mbox{and}\quad P_a(v):=2a\circ (a\circ v)-a^2\circ v\quad (v\in \V).$$
These self-adjoint linear transformations appear prominently in the study of Euclidean Jordan algebras. 
The norms of these transformations relative to some spectral norms have been described in \cite{gowda-holdertype}. For $r,s\in [1,\infty]$, we have (see \cite{gowda-holdertype})
$$||a||_\infty\leq ||L_a||_{r\rightarrow s}\quad\mbox{and}\quad ||a^2||_\infty=||a||_\infty^2\leq ||P_a||_{r\rightarrow s}.$$
Additionally, for any $p\in [1,\infty]$ with conjugate $q$, 
\begin{itemize}
\item [$\bullet$] $||L_a||_{p\rightarrow p}=||L_a||_{p\rightarrow \infty}=||L_a||_{1\rightarrow q}=||a||_\infty$  and 
$||L_a||_{p\rightarrow 1}=||L_a||_{\infty\rightarrow q}=||a||_q$,
\item [$\bullet$] 
$||P_a||_{p\rightarrow p}=||P_a||_{p\rightarrow \infty}=||P_a||_{1\rightarrow q}=||a||_\infty^2$  and
$||P_a||_{p\rightarrow 1}=||P_a||_{\infty\rightarrow q}=||a^2||_q$.
\end{itemize}

We now come to the estimation of $||L_a||_{r\rightarrow s}$ and  $||P_a||_{r\rightarrow s}$ for $r\neq s$. First suppose  $1\leq r< s\leq \infty$. Then, using the above properties and the fact that for any  $x\in\Rn$, $||x||_p$ is a decreasing function of $p$ over $[1,\infty]$, we  have
$$||a||_\infty\leq ||L_a||_{r\rightarrow s}=\sup_{0\neq v\in \V}\frac{||L_a(v)||_s}{||v||_r}\leq 
\sup_{0\neq v\in \V}\frac{||L_a(v)||_r}{||v||_r}=||L_a||_{r\rightarrow r}=||a||_\infty.$$
Thus,
$$||L_a||_{r\rightarrow s}=||a||_\infty\,\,\,(1\leq r<s\leq \infty).$$
A similar argument shows that 
$$||P_a||_{r\rightarrow s}=||a||_\infty^2\,\,\,(1\leq r<s\leq \infty).$$
When $1\leq s<r\leq \infty$, 
Corollary \ref{corollary} yields  the following estimate:
$$||L_a||_{r\rightarrow s}\leq 2\sqrt{2}\,||a||_{(\frac{r}{s})^\prime}.
$$
For the same $s$ and $r$, we can get a different estimate
\begin{equation}\label{different estimate}
||L_a||_{r\rightarrow s}\leq 2\,C_q\,||a||_{p},
\end{equation}
where $\frac{1}{p}=\frac{1}{s}-\frac{1}{r}$ (so that $p=s(\frac{r}{s})^\prime$) and $q$ is the conjugate of $p$.
To see this, we apply Theorem  
\ref{interpolation theorem} with
$$r_0=\infty,\,s_0=p,\,r_1=q,\,s_1=1,\,r_\theta=r,\,s_\theta=s,\,\mbox{and}\,\,\theta=\frac{q}{r}.$$ Then,
$$||L_a||_{r\rightarrow s} \leq C\,||L_a||^{1-\theta}_{\infty\rightarrow p}\,||L_a||^{\theta}_{q\rightarrow 1}=C\,||a||_p,$$
where $C=\max\{C_{r_0}C_{s_0^\prime},\,C_{r_1}C_{s_1^\prime}\}=2\,C_q$. 
To see an interesting consequence of (\ref{different estimate}), let
$1\leq r,s,p\leq \infty$ with $r\neq s$ and $\frac{1}{s}=\frac{1}{p}+\frac{1}{r}.$ Then, using the inequality $||a\circ b||_{s}\leq ||L_a||_{r\rightarrow s}\,||b||_r$, the  estimate (\ref{different estimate}) leads to
$$||a\circ b||_{s}\leq 2\,C_q\,||a||_p\,||b||_r\quad (a,b\in \V),$$
which can be regarded as a {\it generalized H\"{o}lder type inequality}. We remark that the special case  $s=1$ was already covered in Theorem \ref{some p,q norm inequalities} with $1$ in place of $2C_q$. It is very likely that the inequality $||a\circ b||_{s}\leq ||a||_p\,||b||_r$ holds in the general case as well.

Analogous to the above norm estimates of $L_a$, we can 
estimate $||P_a||_{r\rightarrow s}$ when $r>s$ (with $p$ and $q$ defined above):
$$||P_a||_{r\rightarrow s}\leq 2\sqrt{2}\,
||a^2||_{(\frac{r}{s})^\prime}\quad\mbox{and}\quad
||P_a||_{r\rightarrow s}\leq 2\,C_q
\,||a^2||_{p}.$$

We now consider a {\it positive linear transformation} $P$ on $\V$, which is a linear transformation on $\V$ satisfying the condition
$$a\geq 0\Rightarrow P(a)\geq 0,$$
where $a\geq 0$ means that $a$ belongs to the symmetric cone of $\V$ (or, equivalently, it is the square of some element of $\V$). Examples of such transformations include: 

\begin{itemize}
\item [$\bullet$] Any nonnegative matrix on the algebra $\Rn$.
\item [$\bullet$] Any quadratic representation $P_a$ on $\V$ \cite{faraut-koranyi}.
\item [$\bullet$] The transformation $P_A$ defined on $\Sn$ (the algebra of $n\times n$ real symmetric matrices)  by $P_A(X)=AXA^T$, where $A\in \R^{n\times n}$.
\item [$\bullet$] The transformation $P=L^{-1}$ on $\V$, where $L:\V\rightarrow \V$ is linear, positive stable (which means that all eigenvalues of $L$ have positive real parts) and satisfies the $Z$-property:
$$a\geq 0,b\geq 0,\,\langle a,b\rangle =0\Rightarrow \langle L(a),b\rangle \leq 0.$$ In particular,
on the algebra $\Hn$ (of $n\times n$ complex Hermitian matrices),
$P=L_{A}^{-1}$, where $A$ is a complex $n\times n$  positive stable matrix and  $L_A(X):=AX+XA^*$.
\item [$\bullet$] Any  {\it doubly stochastic transformation} on $\V$ \cite{gowda positive map}: It is a positive linear transformation $P$ with $P(e)=e=P^*(e)$. 
\end{itemize}

For any positive linear transformation $P$ on $\V$, and $p\in [1,\infty]$ with conjugate $q$, we have the following from \cite{gowda-holdertype}:
\begin{itemize}
\item [$(i)$] $||P||_{\infty\rightarrow p}= ||P(e)||_p$ and $||P||_{p\rightarrow 1}= ||P^*(e)||_q$.
\item [$(ii)$] $||P||_{p\rightarrow \infty}\leq ||P(e)||_{\infty}$ and $||P||_{1\rightarrow p}\leq ||P^*(e)||_{\infty}$.
\item [$(iii)$]  $||P||_{p\rightarrow p}\leq ||P(e)||_{\infty}^{1-\frac{1}{p}}\,||P^*(e)||_{\infty}^{\frac{1}{p}}.$
\end{itemize}

So, for a positive $P$, an application of Corollary \ref{corollary} gives the following inequalities: 
\begin{itemize}
\item [$(i)$] $||P||_{r\rightarrow s}\leq 2\sqrt{2}||P(e)||_{\infty}^{1-\frac{1}{r}}\,||P^*(e)||_\infty^{\frac{1}{r}}$ when $r<s$. 
\item [$(ii)$] $||P||_{r\rightarrow s}\leq 2\sqrt{2}||P(e)||_{\infty}^{1-\frac{1}{s}}\,||P^*(e)||_{(\frac{r}{s})^\prime}^{\frac{1}{s}}$ when $r>s$.
\end{itemize}
Additionally, when $P$ is also self-adjoint  and $r>s$, analogous  to (\ref{different estimate}), one can get the following estimate:
$$||P||_{r\rightarrow s}\leq 2\,C_q\,||P(e)||_p.$$

\section{Appendix}

\begin{proposition} \label{optimization} 
{\it Given $p\in [1,\infty]$ with conjugate $q$, consider the following  real valued functions defined over $\R^n\times \Rn$, $n\geq 2$:
$$f(x,y)=||x||_p+||y||_p\quad\mbox{and}\quad g(x,y)=||x+iy||_p.$$
Then, 
\begin{equation} \label{prop1}
\max\Big \{f(x,y):\, g(x,y)=1\}=C_p,
\end{equation}
where 
$$C_p=\left \{\begin{array}{cl}  \sqrt{2} & {\rm if}~ 1 \leq p <2,\\2^{\frac{1}{q}} & {\rm if}~ 2\leq p \leq \infty\,. \end{array} \right . $$
}
\end{proposition}

\vspace{.1cm}
\begin{proof} By the continuity of $f$ and $g$, and the compactness of the constraint set, the maximum in (\ref{prop1}) 
is attained.
\\
It is easy to see that the pair $(\overline{x},\overline{y})$ with $\overline{x}=2^{-\frac{1}{p}}(1,0,0\ldots,0)$ 
and $\overline{y}=2^{-\frac{1}{p}}(0,1,0\ldots,0)$ satisfies the (constraint) equation $g(x,y)=1$. Hence 
\begin{equation} \label{geq inequality for cp}
C_p\geq f(\overline{x},\overline{y})=2^{\frac{1}{q}}.
\end{equation}
Consider any pair $(x,y)\in \Rn\times \Rn$ with $g(x,y)=1$. Writing $x=(x_1,x_2,\ldots, x_n)$, etc., by H\"{o}lder's inequality, we have
\begin{equation} \label{intermediate inequality}
||x||_p+||y||_p\leq 2^{\frac{1}{q}}\Big ( ||x||^p_p+||y||^p_p\Big )^{\frac{1}{p}}=2^{\frac{1}{q}}\Big ( \sum_{j=1}^{n}(|x_j|^p+|y_j|^p)\Big )^{\frac{1}{p}}.
\end{equation}
We consider three cases.\\
\noindent{\it Case 1:} $p=\infty$.
By (\ref{geq inequality for cp}), $C_\infty\geq 2^{\frac{1}{q}}=2$ (as $q=1$).
Since $|x_j+iy_j|\leq 1$ for all $j$ (from our constraint), we get $||x||_\infty,||y||_\infty \leq 1$; hence $C{_\infty} \leq 2$. 
We conclude that  $C_{\infty}=2$.
\\

\noindent{\it Case 2:} $2\leq p< \infty$.\\
In this case, we use the well-known Clarkson inequality for complex numbers $z$ and $w$ (see \cite{alrimawi et al}, page 163): 
$$2(|z|^p+|w|^p)\leq |z+w|^p+|z-w|^p.$$
 Then, for each $j$, with $z=x_j$ and $w=iy_j$, we have
$$2\big (|x_j|^p+|y_j|^p\big )\leq |x_j+iy_j|^p+|x_j-iy_j|^p.$$
Summing over $j$ and noting $|x_j+iy_j|=|x_j-iy_j|$, we get
$$\sum_{j=1}^{n}\big (|x_j|^p+|y_j|^p\big )\leq \sum_{j=1}^{n}|x_j+iy_j|^{p}=g(x,y)^p=1.$$ 
It follows from (\ref{intermediate inequality}) that $||x||_p+||y||_p\leq 2^{\frac{1}{q}}.$ As this holds for all $(x,y)$ with $g(x,y)=1$, we have 
$C_p\leq 2^{\frac{1}{q}}.$ From (\ref{geq inequality for cp}) 
we conclude that $C_p=2^{\frac{1}{q}}.$
\\

\noindent{\it Case 3:} $1\leq p< 2$.\\
Let $\delta:=n^{-\frac{1}{p}}2^{-\frac{1}{2}}.$  
It is easy to see that the pair $(\overline{x},\overline{y})$ with  $\overline{x}=\delta(1,1,\ldots,1)=\overline{y}$ 
satisfy the constraint equation $g(x,y)=1$. As $f(\overline{x},\overline{y})=\sqrt{2}$ we have, $C_p\geq \sqrt{2}.$
\\
Now,  as $1\leq p< 2$, we use a refined version of  Clarkson inequality presented in \cite{alrimawi et al}, Theorem 2.3:
$$2^{p-1}(|z|^p+|w|^p)+(2-2^{\frac{p}{2}})\min\{|z+w|^p, |z-w|^p\}\leq |z+w|^p+|z-w|^p.$$
Then, for each $j$, with $z=x_j$ and $w=iy_j$, we have 
$$2^{p-1}\big (|x_j|^p+|y_j|^p\big )+(2-2^{\frac{p}{2}})\min\{|x_j+iy_j|^p, |x_j-iy_j|^p\}\leq |x_j+iy_j|^p+|x_j-iy_j|^p.$$
Simplifying this expression and summing over $j$, we get
$$\sum_{j=1}^{n}\big(|x_j|^p+|y_j|^p\big ) \leq 2^{1-\frac{p}{2}}\Big (\sum_{j=1}^{n}|x_j+iy_j|^p\Big )=2^{1-\frac{p}{2}}g(x,y)^p=2^{1-\frac{p}{2}}.$$
\\
This leads, via (\ref{intermediate inequality}), to
$$||x||_p+||y||_p\leq 2^{\frac{1}{q}}\,(2^{1-\frac{p}{2}})^{\frac{1}{p}}=\sqrt{2}.$$
Now, taking the maximum of $||x||_p+||y||_p$ over $(x,y)$, we get $C_p\leq \sqrt{2}.$
Thus, when $1\leq p< 2$,
$$C_p =\sqrt{2}.$$
This completes our proof.
\end{proof}


\end{document}